\documentclass[12pt]{amsart}
\usepackage{amsmath,amssymb,amsbsy,amsfonts,latexsym,amsopn,amstext,
                                               amsxtra,euscript,amscd,bm}
\usepackage{url}
\usepackage[colorlinks,linkcolor=blue,anchorcolor=blue,citecolor=blue]{hyperref}
\usepackage{color}
\usepackage{float}

\hypersetup{breaklinks=true}

\begin{document}

\newtheorem{theorem}{Theorem}
\newtheorem{lemma}[theorem]{Lemma}
\newtheorem{claim}[theorem]{Claim}
\newtheorem{cor}[theorem]{Corollary}
\newtheorem{prop}[theorem]{Proposition}
\newtheorem{definition}{Definition}
\newtheorem{question}[theorem]{Question}
\newcommand{\hh}{{{\mathrm h}}}

\numberwithin{equation}{section}
\numberwithin{theorem}{section}
\numberwithin{table}{section}

\def\sssum{\mathop{\sum\!\sum\!\sum}}
\def\ssum{\mathop{\sum\ldots \sum}}
\def\iint{\mathop{\int\ldots \int}}

\def\squareforqed{\hbox{\rlap{$\sqcap$}$\sqcup$}}
\def\qed{\ifmmode\squareforqed\else{\unskip\nobreak\hfil
\penalty50\hskip1em\null\nobreak\hfil\squareforqed
\parfillskip=0pt\finalhyphendemerits=0\endgraf}\fi}%%

%  use the AMS-Euler Fraktur fonts
%%%%%%%%%%%%%%%%%%%%%%%%%%%%%%%%%%
\newfont{\teneufm}{eufm10}
\newfont{\seveneufm}{eufm7}
\newfont{\fiveeufm}{eufm5}
%%%%%%%%%%%%%%%%%%%%%%%%%%%%%%%%%
%
%  allow automatic size selection in math mode
%
%%%%%%%%%%%%%%%%%%%%%%%%%%%%%%%%%
\newfam\eufmfam
     \textfont\eufmfam=\teneufm
\scriptfont\eufmfam=\seveneufm
     \scriptscriptfont\eufmfam=\fiveeufm
%%%%%%%%%%%%%%%%%%%%%%%%%%%%%%%%%
%
%  \frak works on a single symbol at a time...
%
\def\frak#1{{\fam\eufmfam\relax#1}}

\newcommand{\bflambda}{{\boldsymbol{\lambda}}}
\newcommand{\bfmu}{{\boldsymbol{\mu}}}
\newcommand{\bfxi}{{\boldsymbol{\xi}}}
\newcommand{\bfrho}{{\boldsymbol{\rho}}}

\def\fK{\mathfrak K}
\def\fT{\mathfrak{T}}

\def\fA{{\mathfrak A}}
\def\fB{{\mathfrak B}}
\def\fC{{\mathfrak C}}

\def \balpha{\bm{\alpha}}
\def \bbeta{\bm{\beta}}
\def \bgamma{\bm{\gamma}}
\def \blambda{\bm{\lambda}}
\def \bchi{\bm{\chi}}
\def \bphi{\bm{\varphi}}
\def \bpsi{\bm{\psi}}

\def\eqref#1{(\ref{#1})}

\def\vec#1{\mathbf{#1}}

%\def\squareforqed{\hbox{\rlap{$\sqcap$}$\sqcup$}}
%\def\qed{\ifmmode\squareforqed\else{\unskip\nobreak\hfil
%\penalty50\hskip1em\null\nobreak\hfil\squareforqed
%\parfillskip=0pt\finalhyphendemerits=0\endgraf}\fi}

%%%%%%%%%%%%%%%%%%%%%%%%%
% Alphabet calligraphie %
%%%%%%%%%%%%%%%%%%%%%%%%%
\def\cA{{\mathcal A}}
\def\cB{{\mathcal B}}
\def\cC{{\mathcal C}}
\def\cD{{\mathcal D}}
\def\cE{{\mathcal E}}
\def\cF{{\mathcal F}}
\def\cG{{\mathcal G}}
\def\cH{{\mathcal H}}
\def\cI{{\mathcal I}}
\def\cJ{{\mathcal J}}
\def\cK{{\mathcal K}}
\def\cL{{\mathcal L}}
\def\cM{{\mathcal M}}
\def\cN{{\mathcal N}}
\def\cO{{\mathcal O}}
\def\cP{{\mathcal P}}
\def\cQ{{\mathcal Q}}
\def\cR{{\mathcal R}}
\def\cS{{\mathcal S}}
\def\cT{{\mathcal T}}
\def\cU{{\mathcal U}}
\def\cV{{\mathcal V}}
\def\cW{{\mathcal W}}
\def\cX{{\mathcal X}}
\def\cY{{\mathcal Y}}
\def\cZ{{\mathcal Z}}
\newcommand{\rmod}[1]{\: \mbox{mod} \: #1}

\def\cg{{\mathcal g}}

\def\vr{\mathbf r}

\def\e{{\mathbf{\,e}}}
\def\ep{{\mathbf{\,e}}_p}
\def\eq{{\mathbf{\,e}}_q}

\def\em{{\mathbf{\,e}}_m}

\def\Tr{{\mathrm{Tr}}}
\def\Nm{{\mathrm{Nm}}}

 \def\SS{{\mathbf{S}}}

\def\lcm{{\mathrm{lcm}}}

\def\({\left(}
\def\){\right)}
\def\l|{\left|}
\def\r|{\right|}
\def\fl#1{\left\lfloor#1\right\rfloor}
\def\rf#1{\left\lceil#1\right\rceil}
\def\flq#1{\langle #1 \rangle_q}
\def\flp#1{\langle #1 \rangle_p}

\def\mand{\qquad \mbox{and} \qquad}

\newcommand{\commK}[1]{\marginpar{%
\begin{color}{red}
\vskip-\baselineskip %raise the marginpar a bit
\raggedright\footnotesize
\itshape\hrule \smallskip K: #1\par\smallskip\hrule\end{color}}}

\newcommand{\commI}[1]{\marginpar{%
\begin{color}{magenta}
\vskip-\baselineskip %raise the marginpar a bit
\raggedright\footnotesize
\itshape\hrule \smallskip I: #1\par\smallskip\hrule\end{color}}}

\newcommand{\commT}[1]{\marginpar{%
\begin{color}{blue}
\vskip-\baselineskip %raise the marginpar a bit
\raggedright\footnotesize
\itshape\hrule \smallskip T: #1\par\smallskip\hrule\end{color}}}

%%%%%%%%%%%%%%%%%%%%%%%%%%%%%%%%%%%%%%%%%%%%%%%%%%%%%%%%
%%%%%%%%%%%%%%%%%%%%%%%%%%%%%%%%%%%%%%%%%%%%%%%%%%%%%%%%
%%%%%%%%%%%%%%%%%%%%%%%%%%%%%%%%%%%%%%%%%%%%%%%%%%%%%%%%
%%%%%%%%%%%%%%%%%%%%%%%%%%%%%%%%%%%%%%%%%%%%%%%%%%%%%%%%

%%%%%%%  END OF STANDARD STUFF %%%%%%%%%

%%%%%%%%%%%%%%%%%%%%%%%%%%%%%%%%%%%%%%%%%%%%%%%%%%%%%%%%
%%%%%%%%%%%%%%%%%%%%%%%%%%%%%%%%%%%%%%%%%%%%%%%%%%%%%%%%
%%%%%%%%%%%%%%%%%%%%%%%%%%%%%%%%%%%%%%%%%%%%%%%%%%%%%%%%
%%%%%%%%%%%%%%%%%%%%%%%%%%%%%%%%%%%%%%%%%%%%%%%%%%%%%%%
%%%%%%%%%%%
%%% Spell

\hyphenation{re-pub-lished}

\mathsurround=1pt

\def\bfdefault{b}
\overfullrule=5pt

\def \F{{\mathbb F}}
\def \K{{\mathbb K}}
\def \Z{{\mathbb Z}}
\def \Q{{\mathbb Q}}
\def \R{{\mathbb R}}
\def \C{{\\mathbb C}}
\def\Fp{\F_p}
\def \fp{\Fp^*}

\def\Smn{S_{k,\ell,q}(m,n)}

\def\Kmn{\cK_p(m,n)}
\def\psmn{\psi_p(m,n)}

\def\SM{\cS_{k,\ell,q}(\cM)}
\def\SMN{\cS_{k,\ell,q}(\cM,\cN)}
\def\SAMN{\cS_{k,\ell,q}(\cA;\cM,\cN)}
\def\SABMN{\cS_{k,\ell,q}(\cA,\cB;\cM,\cN)}

\def\SIJq{\cS_{k,\ell,q}(\cI,\cJ)}
\def\SAJq{\cS_{k,\ell,q}(\cA;\cJ)}
\def\SABJq{\cS_{k,\ell,q}(\cA, \cB;\cJ)}

\def\sM{\cS_{k,q}^*(\cM)}
\def\sMN{\cS_{k,q}^*(\cM,\cN)}
\def\sAMN{\cS_{k,q}^*(\cA;\cM,\cN)}
\def\sABMN{\cS_{k,q}^*(\cA,\cB;\cM,\cN)}

\def\sIJq{\cS_{k,q}^*(\cI,\cJ)}
\def\sAJq{\cS_{k,q}^*(\cA;\cJ)}
\def\sABJq{\cS_{k,q}^*(\cA, \cB;\cJ)}
\def\sABJp{\cS_{k,p}^*(\cA, \cB;\cJ)}

 \def \xbar{\overline x}

\title[Exponential Sums with Binomials]{Bilinear Forms with Exponential Sums with Binomials}

 \author[K. Liu] {Kui Liu}
\address{School of Mathematics and  Statistics, Qingdao University, No.308, Ningxia Road, Shinan, Qingdao, Shandong, 266071, P. R. China}
\email{liukui@qdu.edu.cn}

 \author[I. E. Shparlinski] {Igor E. Shparlinski}

\address{Department of Pure Mathematics, University of New South Wales,
Sydney, NSW 2052, Australia}
\email{igor.shparlinski@unsw.edu.au}

 \author[T. P. Zhang] {Tianping Zhang}

\thanks{T. P. Zhang is the corresponding author (tpzhang@snnu.edu.cn).}

\address{School of Mathematics and Information Science, Shaanxi Normal University, Xi'an 710019 Shaanxi, P. R. China}
\email{tpzhang@snnu.edu.cn}

\begin{abstract} We obtain several  estimates for bilinear form with  exponential  sums
with binomials $mx^k + nx^\ell$. In particular we show the existence of nontrivial cancellations between
such sums when the coefficients $m$ and $n$ vary over rather sparse sets of general nature.
 \end{abstract}

\keywords{Binomial sums, cancellation, bilinear form}
\subjclass[2010]{11D79, 11L07}

\maketitle

\section{Introduction}
\label{sec:intro}

\subsection{Background and motivation}

For    a positive
integer $q$, we denote  by $\Z_q$ the residue ring modulo $q$  and also denote by $\Z_q^*$ the group
of units of $\Z_q$.

For fixed integers $k$ and $\ell$, we consider exponential sums with binomials
$$
\Smn = \sum_{x\in \Z_q^*} \eq\(mx^k +nx^\ell\),
$$
where for negative $k$ or $\ell$  the  inversion of $x$ is computed  modulo $q$ and
$$
\eq(z) = \exp(2 \pi i z/q).
$$
The case $(k,\ell) = (1,-1)$ corresponds to the case of Kloosterman sums.

Furthermore, given two sets $\cM,\cN \subseteq \Z_q$
%$$
%\cI = [K+1, K+M],\ \cJ = [L+1, L+N] \subseteq [1, p-1],
%$$
and  two sequences of weights $\cA = \{\alpha_m\}_{m\in \cM}$ and $\cB = \{\beta_n\}_{n\in \cN}$, we define the bilinear sums of  binomial sums
$$
\SABMN = \sum_{m\in \cM} \sum_{n \in \cN} \alpha_m \beta_n \Smn .
$$
We also consider the following special cases
 \begin{equation}
 \label{eq:Special SAMN}
\begin{split}
 \SAMN&= \cS_{k,\ell,q}\(\cA, \{1\}_{n\in \cN};\cM,\cN\)\\
 & = \sum_{m\in \cM} \sum_{n \in \cN} \alpha_m\Smn,
 \end{split}
\end{equation}
 and
  \begin{equation}
 \label{eq:Special SMN}
\begin{split}
\SMN&= \cS_{k,\ell,q}\(\{1\}_{m\in \cM}, \{1\}_{n\in \cN} ;\cM,\cN\)\\
&= \sum_{m\in \cM} \sum_{n \in \cN} \Smn.
\end{split}
\end{equation}

For $(k,\ell) = (1,-1)$, that is, for Kloosterman sums, such bilinear forms have been introduced
by Fouvry,   Kowalski and Michel~\cite{FKM} who have also demonstrated the importance
of estimating them beyond of what follows immediately from the Weil bound, see~\cite[Chapter~11]{IwKow},
which is essentially given by~\eqref{eq:trivial} below.

More generally, for arbitrary modulus $q$ and exponents $(k,\ell)$ one can apply the
general bound of~\cite[Theorem~1]{Shp1} on exponential sums with
few nomials to derive
 \begin{equation}
\label{eq:trivial}
\left| \SABMN  \right| \le  MN  q^{1/2+o(1)}\max_{m\in \cM}  |\alpha_m|  \max_{n \in \cN} | \beta_n |,
\end{equation}
to which we refer as the {\it trivial bound\/}.

Further progress in the case $(k,\ell) = (1,-1)$ has been
achieved in~\cite{BFKMM1,BFKMM2,KMS,Shp2,ShpZha}. In~\cite{Xi-FKM}  this question has been studied
on average over the moduli $q$.
We also recall recent results of~\cite{Khan,LSZ,WuXi} when cancellations among
Kloosterman sums are studied for moduli of special arithmetic structure.
Furthermore, in
the case of a prime $q=p$ and  $(k,\ell) = (2,-1)$ has been studied by Nunes~\cite{Nun1},
via the method of Fouvry,   Kowalski and Michel~\cite{FKM}. Then these sums have been
used to investigate the distribution of squarefree integers in arithmetic progressions;
see Section~\ref{sec:comp} for exact formulations of the results of Nunes~\cite{Nun1}
and their comparison with our bounds.

We remark that the method introduced by  Fouvry,   Kowalski and Michel~\cite{FKM},
and then further developed and  used in~\cite{BFKMM1,KMS,Nun1},
relies heavily on  such deep tools as the Weil and Deligne bounds, see~\cite[Chapter~11]{IwKow}.
In particular, this approach works well only for prime moduli $p$.
It is important to note the methods of~\cite{Shp2,ShpZha} are of elementary nature, and
in particular work without any losses of strength for composite $q$ as well. On the other hand, the method
of~\cite{BFKMM1,FKM,KMS,Xi-FKM}
works for much more general objects than Kloosterman and other similar exponential  sums.

\subsection{General notation}

%We always assume that  the sequence of weights $\cA = \{\alpha_m\}_{m\in \cI}$ is
%supported only on $m$ with $\gcd(m,q)=1$, that is, we have $\alpha_m = 0$  if $\gcd(m,q)>1$.
We remark that our bounds involve  only  the norms of the weights $\cA$
but do not explicitly depend on the size of the set $\cM$ on which they are supported.
Hence, without loss of generality, we can assume that $\cM = \Z_q$.  On the other hand,
our method does not apply to general sets $\cN$ and works only when $\cN$ is an interval,
and thus, for the sums with weights we simplify the notation as
 \begin{equation}
  \label{eq:Simple SABJ}
\SABJq = \sum_{m\in  \Z_q} \sum_{n \in \cJ} \alpha_m  \beta_n\Smn,
\end{equation}
and even further as
 \begin{equation}
\label{eq:Simple SAJ}
\SAJq = \sum_{m\in  \Z_q} \sum_{n \in \cJ} \alpha_m \Smn,\\
%& \SIJq = \sum_{m\in  \Z_q} \sum_{n \in \cJ}  \Smn,\?
\end{equation}
where $\cJ = \{L+1, \ldots, L+N\} \subseteq \Z_q$ is a set of $N$ consecutive residues
of $\Z_q$ (with $q-1$ followed by $0$).  Furthermore,  in the case of the sums without weights
we only estimate such sums when the set $\cM = \cI  =  \{K+1,\ldots, K+M\} \subseteq \Z_q$ is another interval, and thus
we write  $\SIJq$.

The case of $\ell = -1$ is somewhat special as it admist some extra treatment and
is also important for many applications, see~\cite{Nun1} for example. Thus we introduce special
notation
 \begin{equation}
\label{eq:Simple l=-1}
\begin{split}
& \sABJq = \cS_{k,-1,q}(\cA, \cB;\cJ), \\
&\sAJq = \cS_{k,-1,q}(\cA;\cJ), \\
& \sIJq = \cS_{k,-1,q}(\cI,\cJ).
\end{split}
\end{equation}

For an integer $u$ we define
$$
\flq{u} = \min_{k \in \Z} |u - kq|
$$
as the distance to the closest integer,  which is a multiple of $q$.

We also define the norms
 $$
 \|\cA\|_\infty=\max_{m\in \cM}|\alpha_m|  \mand \|\cA\|_\sigma =\( \sum_{m\in \cM} |\alpha_m|^\sigma\)^{1/\sigma},
 $$
where  $\sigma >0$, and similarly for the weights $\cB$.

Throughout the paper,  as usual $A\ll B$  is equivalent to the inequality $|A|\le cB$
with some  constant $c>0$, which may depend on the integers $k$ and $\ell$, and occasionally, where obvious,
the real parameter $\varepsilon>0$ and on the integer parameter $\nu \ge 1$.

The letter $p$ always denotes a prime number and we say that $q$ is {\it squarefree\/} if
it not divisible by $p^2$ for any $p$.

\section{New results}

\subsection{Bounds for every $q$}
\label{eq:every q}

We start with the sums $\SAJq$, defined in~\eqref{eq:Simple SAJ}, which are medium level of complexity as
one variable still runs through a continuous interval.
The proof is based on the method from~\cite{ShpZha}, coupled with a result of
Pierce~\cite[Theorem~4]{Pierce}

\begin{theorem}
\label{thm:SAJq}     If $k\ne \ell$ are  fixed nonzero integers with
$\gcd(k,\ell)=1$, then, for any fixed positive integer $\nu$, squarefree $q\ge 1$ and
$$
 \cJ = \{L+1, \ldots, L+N\} \subseteq \Z_q,
$$
we have
\begin{align*}
&\SAJq \\
& \qquad \ll   \min\Bigl\{ \|\cA\|_2  N^{1/2} q,\,  \|\cA\|_{1}^{1-1/\nu} \|\cA \|_{2}^{1/\nu}\\
& \qquad \qquad \qquad \qquad \qquad \qquad \quad\(q+q^{(2\nu^2 + \nu +1)/2 \nu (\nu+1)}
  N^{1/(\nu+1)}\)q^{o(1)} \Bigr\}.
\end{align*}
\end{theorem}

In particular, when $\alpha_m$ is the characteristic function of the interval
$\cI = \{K+1,\ldots, K+M\}$ we obatin a bound on the sums
 $\SIJq$, defined by~\eqref{eq:Special SMN}. We also see that in the case
 of the sums $ \SIJq $ the roles of $M$ and $N$ can be interchanged.

 \begin{cor}
  \label{cor:SIJq}    If $k\ne \ell$ are  fixed nonzero integers with $\gcd(k,\ell)=1$,
then, for any  fixed positive integer $\nu$,  squarefree $q\ge 1$ and
$$
\cI = \{K+1,\ldots, K+M\},\ \cJ = \{L+1, \ldots, L+N\} \subseteq \Z_q,
$$
we have
$$
% \SIJq  \le \(qN^{1-1/2\nu} + q^{(2\nu^2 + \nu +1)/2 \nu (\nu+1)} M^{1/(\nu+1)}N^{1-1/2\nu} \)
%q^{o(1)}.
 \SIJq  \le X^{1-1/2\nu}  \(q+ q^{(2\nu^2 + \nu +1)/2 \nu (\nu+1)} Y^{1/(\nu+1)} \)
q^{o(1)}
$$
for any choice of $X,Y$ with $\{X,Y\} = \{M,N\}$.
\end{cor}

In particular, with $\nu =2$ we obtain from Corollary~\ref{cor:SIJq} that
 \begin{equation}
\label{eq:Simple Bound}
 \SIJq \le q^{1+o(1)}N^{3/4} + q^{11/12+o(1)}M^{1/3} N^{3/4} .
\end{equation}
This improves the trivial bound~\eqref{eq:trivial} provided that
 \begin{equation}
\label{eq:Simple Bound Range}
M^4N \ge q^{2+\varepsilon}  \mand M^{8} N^{3} \ge q^{5+\varepsilon}
\end{equation}
for some fixed $\varepsilon > 0$, and in particular for   $M=N \ge q^{5/11+\varepsilon}$.
Note that in~\eqref{eq:Simple Bound} and~\eqref{eq:Simple Bound Range} the roles of
$M$ and $N$ can be interchanged.

In the case $M,N =  q^{1/2+o(1)}$ crucial for many applications, Corollary~\ref{cor:SIJq}
implies the bound $MNq^{1/2-1/24+o(1)}$, saving $q^{1/24}$ compared to the trivial
bound~\eqref{eq:trivial}.

Finally, we estimate the sum $\sABJq$, see~\eqref{eq:Simple SABJ} and~\eqref{eq:Simple l=-1}, which is the most complicated case that requires some extra arguments combined with the
ideas of~\cite{Shp2,ShpZha}.

\begin{theorem}
\label{thm:sABJp}     If $k>1$ is a fixed integer,
then, for any  fixed positive integer $\nu$, prime $p\ge 1$ and
$$
 \cJ = \{L+1, \ldots, L+N\} \subseteq \Z_p,
$$
we have
$$
 \sABJp  \le   \sqrt{\|\cA\|_2\|\cB\|_2}\(p^{(6\nu -1)/4\nu}+p^{(3\nu+2)/2(\nu+1)}N^{1/2(\nu+1)}\) p^{o(1)}.
$$
\end{theorem}

\subsection{Bounds for almost all $q$}
\label{eq:almost all q}

We also show that  in the case of $\ell =1$ for almost all $q$ in a dyadic interval
$ [Q,2Q]$  stronger versions of the results of
Section~\ref{eq:every q} hold.

We also have an analogue of the second bound in Theorem~\ref{thm:SAJq}, but only for
the sums $ \sAJq$.

\begin{theorem}
\label{thm:SAJq Aver}   If $k>1$ is a fixed integer,
then, for any  fixed positive integer $\nu$, fixed real
$\varepsilon>0$  and sufficiently large real $Q\ge 1$,
for all but $O(Q^{1-\varepsilon})$
integers $q \in [Q, 2Q]$, we have
$$
 \sAJq\ll   \|\cA\|_{1}^{1-1/\nu} \|\cA \|_{2}^{1/\nu}\(q+q^{(\nu+1)/2\nu}N^{1/2}\)q^{\varepsilon}.
 $$
\end{theorem}

We now have the following version of Corollary~\ref{cor:SIJq} (however this time
we cannot interchange the roles of $M$ and $N$).

\begin{cor}
\label{cor:SIJq Aver}   If $k>1$ is a fixed integer,
then, for any  fixed positive integer $\nu$, fixed real
$\varepsilon>0$ and sufficiently large real $Q\ge 1$,
for all but $O(Q^{1-\varepsilon})$
integers $q \in [Q, 2Q]$, we have
$$
 \sIJq  \ll  M^{1-1/2\nu} \(q + q^{(\nu +1)/2 \nu} N^{1/2} \)  q^{\varepsilon}.
$$
\end{cor}

\section{Preparations}

  \subsection{Linear and bilinear exponential sums}

We need the following well-known simple results.

First we recall the following bound of linear sums~\cite[Bound~(8.6)]{IwKow}.

\begin{lemma}
\label{lem:lin sums}
For any  integers $u$,  $L$ and $N\ge 1$,
we have
$$
\sum_{n=L+1}^{L+N} \eq(nu) \ll \min\left\{N, \frac{q}{\flq{u}}\right \}.
$$
\end{lemma}

We also  need the following well-known result, which dates back to
 Vinogradov~\cite[Chapter~6, Problem~14.a]{Vin}.

\begin{lemma}
\label{lem:DoubleSum} For arbitrary set $\cU, \cV \subseteq \Z_q$ and complex
numbers $\varphi_u$ and $\psi_v$ with
$$
\sum_{u \in \cU} |\varphi_u|^2  \le \varPhi  \mand
\sum_{v \in \cV}|\psi_v|^2 \le \varPsi,
$$
we have
$$
\left|
\sum_{u\in \cU} \sum_{v \in \cV}
\varphi_u \psi_v \eq(uv) \right| \le \sqrt{\varPhi\varPsi q }.
$$
\end{lemma}

  \subsection{Some equations and congruences}
 \label{sec:CongEqReipr}

 We start with  a very simple result on the monomial congruences.

\begin{lemma}
\label{lem:Monom}   If $k$  is a nonzero integer  then for any $a \in \Z_q$
the congruence
$$
x^k \equiv a \pmod q, \qquad x \in \Z_q^*,
$$
has at most $q^{o(1)}$ solutions.
\end{lemma}

\begin{proof}  Clearly we can assume that $a\in \Z_q^*$ as otherwise there is no
solution. Then the discriminant of the polynomial $X^k-a$ has a bounded greatest common
divisor with $q$ and the result follows from the general bound of Huxley~\cite{Hux}.
\end{proof}

We also need several results of Pierce~\cite{Pierce}, which in turn generalises   previous results of Heath-Brown~[Lemma~1]\cite{H-B2} (which corresponds to $\ell = -1$).
We present these results in slightly more general forms (which are however implicitly contained
 in the argument of~\cite{Pierce}.

 For an integer $\nu \ge 1$ and real $U$ let $I_{k,\ell,\nu,q}(U)$  be the number of solutions to
 the system of congruences
  \begin{equation*}
\begin{split}
v_1+ \ldots+  v_\nu & \equiv  v_{\nu+1}+ \ldots+ v_{2\nu}\pmod q,\\
u_i^k \equiv  v_i^\ell& \pmod q, \quad i =1,\ldots, 2\nu,
\end{split}
\end{equation*}
with  $1 \le u_1,  \ldots, u_{2\nu} \le U$ and unrestricted variables $v_1, \ldots,  v_{2\nu} \in \Z_q$.
We have the following slight extension of the bound of Pierce~\cite[Equation~(6.2)]{Pierce}
(which is free of the  restriction $U \le q^{(\nu+1)/2\nu}$).

We recall that all implied constants are allowed to depend on $\nu$.

\begin{lemma}
\label{lem:Ikl}   If $k\ne \ell$ are  fixed nonzero integers with
$\gcd(k,\ell)=1$, then, for any  fixed positive integer $\nu$, squarefree $q\ge 1$  and
$U \le q$   we have
$$
I_{k,\ell,\nu,q}(U) \le \( U^{2\nu}q^{-1} + U^{2\nu^2/(\nu+1)}\)q^{o(1)}.
$$
\end{lemma}

\begin{proof} We use the following inequality given (in a slightly more precise form) in~\cite[Section~6.3]{Pierce}:
$$
I_{k,\ell,\nu,q}(U) \le \( Q^{-1}U^{2\nu-1} + Q^{\nu} U^{\nu}\)q^{o(1)},
$$
holds for any $Q$, satisfying the conditions
$$
Q < U  \mand 8QU \le q.
$$
Thus taking
$$
Q = \min\left\{U^{(\nu-1)/(\nu + 1)}, q/8U\right\},
$$
we obtain the result.
\end{proof}

Furthermore, for $\ell=-1$ and prime $q =p$, Bourgain and Garaev~\cite[Proposition~1]{BouGar}
extend Lemma~\ref{lem:Ikl} to solutions in intervals away from the origin.
 For a real $U$ and $W$ let $I_{k,\nu,q}^*(U,W)$  be the number of solutions to
 the system of congruences
\begin{equation*}
\begin{split}
\frac{1}{u_1^k}+ \ldots+ \frac{1}{u_\nu^k} &\equiv \frac{1}{u_{\nu+1}^k}+ \ldots+\frac{1}{u_{2\nu}^k} \pmod q, \\
W+1 \le &u_1,  \ldots, u_{2\nu} \le W+U.
\end{split}
\end{equation*}
Then, by~\cite[Proposition~1]{BouGar} we have the following estimate

\begin{lemma}
\label{lem:Ikl shift}   If $k>1$ is a fixed integer, then, for any  fixed positive integer $\nu$,
prime  $p\ge 1$  and $U \le p$ we have
$$
I_{k,\nu,p}^*(W,U) \le \( U^{2\nu}p^{-1} + U^{2\nu^2/(\nu+1)}\)p^{o(1)}.
$$
\end{lemma}

%
%For real $U,V\ge 1$, let  $T_{k,\ell,q}(U,V)$ be the number of solutions to
% the  congruence
%$$u^k \equiv \pm v^\ell \pmod q, \qquad
%1 \le u \le U, \ 1 \le v\le V.
%$$
%Using Lemma~\ref{lem:Ikl} in the argument of the proof
%of~\cite[Theorem~4]{Pierce}, we obtain:
%
%\begin{lemma}
%\label{lem:Tkl}    If $k\ne \ell$ are  fixed nonzero integers with
%$\gcd(k,\ell)=1$, then, for any fixed positive integer $\nu$,  squarefree $q\ge 1$
%and real $1 \le U,V \le q$,
%we have
%$$
%T_{k,\ell,q}(U,V) \le \(Uq^{-1/2\nu} + U^{\nu/(\nu+1)}\) V^{1/2\nu} q^{o(1)}.
%$$
%\end{lemma}
%
%Certainly the roles of $U$ and $V$ can be interchanged in the bound of Lemma~\ref{lem:Tkl}.

In the case $\ell=-1$, following our previous convention, we denote
$$
I_{k,\nu,q}^*(U) =  I_{k,-1,\nu,q}(U).
% \mand T_{k,q}^*(U) = T_{k,-1,q}(U).
 $$
We now show that  one can get a better bound on $I_{k,\nu,q}^*(U)$
and thus on $T_{k,q}^*(U)$ on average over $q$ in a dyadic interval $[Q,2Q]$.

Indeed, let $J_{k,\nu}(U)$ be the number
of solutions to the equation
 \begin{equation}
\label{eq:Jk}
\frac{1}{u_1^k}+ \ldots+ \frac{1}{u_\nu^k} = \frac{1}{u_{\nu+1}^k}+ \ldots+\frac{1}{u_{2\nu}^k}, \qquad
1 \le u_1,  \ldots, u_{2\nu} \le U.
  \end{equation}
We have the following bound, which is a slight modification
of a result of Karatsuba~\cite{Kar}, corresponding to $k=1$ and presented in
the proof of~\cite[Theorem~1]{Kar}).

\begin{lemma}
\label{lem:Recipr Eq}   If $k>1$ is a fixed integer,
then, for any fixed positive integer $\nu$,  we have
$$
J_{k,\nu}(U) \le U^{\nu + o(1)}.
$$
\end{lemma}

\begin{proof} Clearing the denominators in~\eqref{eq:Jk} we see that if $p\mid u_i$ for some
component $i =1, \ldots, 2\nu$ of a solution, then we also have $p \mid u_j$ for some $j \ne i$.
This means that for any solution to~\eqref{eq:Jk}, the product $u_1\ldots u_{2\nu}$ is squarefull.
Since any interval $[1, W]$ contains $O(W^{1/2})$ squarefull integers, see~\cite{SS}, applying this with
$W =U^{2\nu}$ and then using the classical  bound on the divisor function, see~\cite[Equation~(1.81)]{IwKow}, we obtain the result.
\end{proof}

Now repeating the argument of the proof of~\cite[Lemma~2.3]{FouShp} and using
Lemma~\ref{lem:Recipr Eq}  in the appropriate place,  we obtain:

\begin{lemma}
\label{lem:Ikl Aver}   If $k>1$ is a fixed integer,
then, for any  fixed positive integer $\nu$ and sufficiently large real $1 \le U \le Q$,
we have
$$
\frac{1}{Q}\sum_{Q \le q \le 2Q} I_{k,\nu,q}^*(U) \le \(U^{2\nu} Q^{-1} + U^\nu\)Q^{o(1)}.
$$
\end{lemma}

Using Lemma~\ref{lem:Ikl Aver}  for every $U_i = e^i$, $ i = 1, \ldots, \rf{\log (2Q)}$,
(where $e=2.7182 \ldots$ is the base of the natural logarithm)
we immediately derive:

\begin{cor}
\label{cor:Ikl Aver}  If $k>1$ is a fixed integer,
then, for any  fixed positive integer $\nu$, real positive $\varepsilon > 0$
and sufficiently large real $Q\ge 1$,  for all but $O(Q^{1-\varepsilon})$
integer $q \in [Q, 2Q]$, we have
$$
 I_{k,\nu,q}^*(U) \le \(U^{2\nu} Q^{-1} + U^\nu\)Q^{\varepsilon+ o(1)}
$$
for every $U \le q$.
\end{cor}

%
%For real $U,V\ge 1$, let  $T_{k,q}(U,V)$ be the number of solutions to
% the  congruence
%$$u^k \equiv \pm v  \pmod q, \qquad
%1 \le u \le U, \ 1 \le v\le V.
%$$
%Now using Corollary~\ref{cor:Ikl Aver}   in the appropriate place of the argument  of the proof
%of~\cite[Theorem~4]{Pierce}, we derive:
%
%\begin{lemma}
%\label{lem:Tkl Aver}   If $k>1$ is a fixed integer,
% then, for any  fixed positive integer $\nu$ real positive $\varepsilon > 0$
%and sufficiently large real $Q\ge 1$ for all but $O(Q^{1-\varepsilon})$
%integer $q \in [Q, 2Q]$, we have
%$$
%T_{k,q}^*(U,V) \le \(Uq^{-1/2\nu} + U^{1/2}\) V^{1/2\nu} q^{\varepsilon/2\nu + o(1)}
%$$
%for  every $U,V \le q $.
%\end{lemma}
%

% \section{Proofs of bounds on sums without weights}

\section{Proofs main results}
\subsection{Proof of Theorem~\ref{thm:SAJq}}

%As in the proof of Theorem~\ref{thm:SIJq},

Changing the order of summation,  we obtain
$$
 \SAJq  = \sum_{x \in\Z_q^*} \sum_{m\in \cI} \alpha_m \eq(mx^k) \sum_{n \in \cJ} \eq(nx^\ell).
$$

Recalling  Lemma~\ref{lem:lin sums}, we obtain
$$
 \SAJq   = \sum_{m\in \Z_q} \sum_{x \in\Z_q^*}\alpha_m
 \gamma_x  \eq(m x^k),
$$
%& \ll \left|\sum_{m\in \cI} \sum_{x \in\Z_q^*}\alpha_m\min\left\{N,  \frac{p}{\flq{x^\ell}}\right \} \eq(m x)\right|\\
where
$$
|\gamma_x| \le  \min\left\{N,  \frac{q}{\flq{x^\ell }}\right \}.
$$

% Similarly to the  proof of Theorem~\ref{thm:SIJq},
We define $I = \rf{\log q}$ and write
 \begin{equation}
\label{eq:BAIJ S0i}
 \SAJq \ll  |\Sigma_0|+ \sum_{i=1}^I |\Sigma_i|,
\end{equation}
where
 \begin{equation*}
\begin{split}
\Sigma_0 &   = \sum_{m\in \Z_q} \sum_{\substack{x \in\Z_q^*\\ \flq{x^\ell} \le q/N}}  \alpha_m
 \gamma_x  \eq(m x^k), \\
\Sigma_i & = \sum_{m\in \Z_q}  \sum_{\substack{x \in\Z_q^*\\  e^{i+1} q/N\ge \flq{x^\ell} > e^iq/N}}
\alpha_m   \gamma_x  \eq(m x^k), \qquad i = 1, \ldots, I.
\end{split}
\end{equation*}

 Now using  Lemmas~\ref{lem:DoubleSum} and~\ref{lem:Monom}, we have
\begin{equation}
\label{eq:S0}
|\Sigma_0| \le  \|\cA\|_2 N\sqrt{ (q/N) q^{1+o(1)} }\le   \|\cA\|_2 N^{1/2}q^{1+o(1)}.
\end{equation}
Also, for $i = 1, \ldots, I$, using that if  $e^{i+1} q/N\ge \flq{x^\ell} > e^iq/N$
then $\gamma_x \ll Ne^{-i}$, hence, again by Lemmas~\ref{lem:DoubleSum} and~\ref{lem:Monom},
we obtain
 \begin{equation*}
\begin{split}
\Sigma_i & =  \sum_{m\in \Z_q} \sum_{\substack{x \in\Z_q^*\\  e^{i+1} q/N\ge \flq{x^k} > e^iq/N}}
\alpha_m   \gamma_x  \eq(m x^\ell) \\
& \le  \|\cA\|_2 (q^{o(1)} N^2e^{-2i} e^{i} q/N)^{1/2} q^{1/2} = e^{-i/2}  \|\cA\|_2 N^{1/2} q^{1+o(1)} .
\end{split}
\end{equation*}
Therefore,
\begin{equation}
\label{eq:Si}
\sum_{i=1}^I |\Sigma_i| \le   \|\cA\|_2 N^{1/2} q^{1+o(1)}  \sum_{i=1}^I e^{-i/2}  \le  \|\cA\|_2 N^{1/2} q^{1+o(1)}.
\end{equation}
Combining~\eqref{eq:S0} and~\eqref{eq:Si},
we obtain the first bound.

For the second bound we turn to use the method of ~\cite{Shp2}. For a fixed integer $\nu\geq 2$,
using the H\"{o}lder inequality, we obtain
\begin{equation}
\begin{split}
\label{eq:S0W0}
|\Sigma_{0}| & \leq \(\sum_{m\in\mathbb{Z}_{q}}|\alpha_{m}|\)^{1-1/\nu}   \(\sum_{m\in\mathbb{Z}_{q}}|\alpha_{m}|^{2}\)^{1/2\nu} W_0\\
& = \|\cA\ \|_{1}^{1-1/\nu} \|\cA\ \|_{2}^{1/\nu}  W_0,
\end{split}
\end{equation}
where
$$
 W_0 =  \(\sum_{m\in\mathbb{Z}_{q}}\left|\sum_{\substack{x \in\Z_q^*\\ \flq{x^\ell} \le q/N}}
\gamma_x  \eq(m x^k)\right|^{2\nu}\)^{1/2\nu}.
$$
%\begin{equation*}
%\begin{split}
%|S_{0}| & \leq  \(\sum_{m\in\mathbb{Z}_{q}}|\alpha_{m}|\)^{1-1/\nu}   \(\sum_{m\in\mathbb{Z}_{q}}|\alpha_{m}|^{2}\)^{1/2\nu}\\
%& \qquad \qquad \qquad \qquad   \(\sum_{m\in\mathbb{Z}_{q}}\left|\sum_{\substack{x \in\Z_q^*\\ \flq{x^\ell}
%\le q/N}} \gamma_x  \eq(m x^k)\right|^{2\nu}\)^{1/2\nu}\\
%& = \|\cA\ \|_{1}^{1-1/\nu} \|\cA\ \|_{2}^{1/\nu}  \(\sum_{m\in\mathbb{Z}_{q}}\left|\sum_{\substack{x \in\Z_q^*\\ \flq{x^\ell} \le q/N}}
%\gamma_x  \eq(m x^k)\right|^{2\nu}\)^{1/2\nu}.
%\end{split}
%\end{equation*}

Opening up the inner sum, changing the order of summation and using the orthogonality of exponential functions, we obtain
 \begin{equation*}
\begin{split}
W_0
& = \sum_{m\in\mathbb{Z}_{q}}
\mathop{\mathop{\sum\cdot\cdot\cdot\sum}_{x_{1},\ldots,x_{2\nu}\in\mathbb{Z}_{q}^{*}}}_{\flq{x_{i}^\ell} \le q/N, i=1,\ldots,2\nu}\prod_{j=1}^{\nu}\gamma_{x_{j}}\overline{\gamma_{x_{\nu+j}}} \eq \(m \sum_{j=1}^{\nu}(x_{j}^k-x_{\nu+j}^{k})\)\\
& = q \mathop{\mathop{\sum\cdot\cdot\cdot\sum}_{\flq{x_{i}^\ell} \le q/N, i=1,\ldots,2\nu}}_{x_{1}^{k}+\ldots+x_{\nu}^{k}\equiv x_{\nu+1}^k+\ldots+x_{2\nu}^{k}\pmod q}\prod_{j=1}^{\nu}\gamma_{x_{j}}\overline{\gamma_{x_{\nu+j}}}\\
& \leq N^{2\nu} q \mathop{\mathop{\sum\cdot\cdot\cdot\sum}_{\flq{x_{i}^\ell} \le q/N, i=1,\ldots,2\nu}}_{x_{1}^{k}+\ldots+x_{\nu}^{k}\equiv x_{\nu+1}^k+\ldots+x_{2\nu}^{k}\pmod q}1.
\end{split}
\end{equation*}

Let $u_{i}=\flq{x_{i}^\ell}$, $v_{i}=x_{i}^{k}$, then we have
 \begin{equation*}
\begin{split}
v_{1}& +\ldots+v_{\nu}\equiv v_{\nu+1}+ \ldots+v_{2\nu}\pmod q,\\
&u_{i}^{k}\equiv \pm v_{i}^{l}\pmod q, \qquad 0<u_{i}\leq q/N.
\end{split}
\end{equation*}
Applying  Lemma~\ref{lem:Ikl}, we have
$$
W_0 \leq N^{2\nu}q^{1+o(1)} \( \(\frac{q}{N}\)^{2\nu}q^{-1}+ \(\frac{q}{N}\)^{2\nu^{2}/(\nu+1)}\).
$$
Then, we see from~\eqref{eq:S0W0}
$$|\Sigma_{0}|\leq \|\cA\ \|_{1}^{1-1/\nu} \|\cA\ \|_{2}^{1/\nu}\(q+q^{(2\nu^2 + \nu +1)/2 \nu (\nu+1)}
  N^{1/(\nu+1)}\)q^{o(1)}.$$
Similarly, we also obtain
$$|\Sigma_{i}|\leq \|\cA\ \|_{1}^{1-1/\nu} \|\cA\ \|_{2}^{1/\nu}\(q+q^{(2\nu^2 + \nu +1)/2 \nu (\nu+1)}
  N^{1/(\nu+1)}e^{-i/(\nu+1)}\)q^{o(1)},$$
and the result now follows from~\eqref{eq:BAIJ S0i}.
%$$\SAJq  \ll \|\cA\ \|_{1}^{1-1/\nu} \|\cA\ \|_{2}^{1/\nu}\(q+q^{(2\nu^2 + \nu +1)/2 \nu (\nu+1)}
%  N^{1/(\nu+1)}\)q^{o(1)}.$$

\subsection{Proof of Theorem~\ref{thm:sABJp}}

By the Cauchy inequality we have
\begin{equation*}
\begin{split}
|\sABJp|^2 & \le \|\cA\|_2  \|\cB\|_2  \sum_{n \in \cJ}
\sum_{m\in \Z_p} \left | \sum_{x\in \Z_p^*} \ep\(mx^k +nx^{-1}\)\right|^2\\
 &=  \|\cA\|_2  \|\cB\|_2  \sum_{n \in \cJ}
\sum_{m\in \Z_p} \left | \sum_{x\in \Z_p^*} \ep\(mx^{-k} +nx\)\right|^2\\
&= \|\cA\|_2  \|\cB\|_2  \sum_{n \in \cJ}
\sum_{m\in \Z_p}  \\
& \qquad \qquad \qquad  \sum_{x,y\in \Z_p^*} \ep\(m(x^{-k}-y^{-k}) +n(x -y)\).
\end{split}
\end{equation*}
Now, writing $y = x+z$ we obtain
\begin{equation*}
\begin{split}
&|\sABJp|^2 \\
&\qquad  \le  \|\cA\|_2  \|\cB\|_2  \sum_{n \in \cJ}
\sum_{m\in \Z_p}  \sum_{x\in \Z_p^*}  \sum_{z\in \Z_p^*-x} \ep\(m(x^{-k}-(x+z)^{-k}) -n z\),
\end{split}
\end{equation*}
where $\Z_p^*-x = \{z \in \Z_p~:~z+x \in \Z_p^*\}$.
Changing the order of summation and applying  Lemma~\ref{lem:lin sums}, we obtain
\begin{equation}
\label{eq:T1}
\begin{split}
\sABJp^2 & \ll \|\cA\|_2  \|\cB\|_2
  \sum_{x\in \Z_p^*} \sum_{m\in \Z_p}   \left| \sum_{z\in\Z_p^*-x}  \eta_z \ep\(m(x+z)^{-k})\)\right|,
\end{split}
\end{equation}
where
$$
|\eta_z| \le  \min\left\{N,  \frac{p}{\flp{z}}\right \}.
$$

For every fixed $x$, to estimate
$$
W(x) =    \sum_{m\in \Z_p}   \left| \sum_{z\in\Z_p^*-x}  \eta_z \ep\(m(x+z)^{-k})\)\right|,
$$
we now set $I=\rf{\log (p/2)}$ and define $2(I+1)$ sets
\begin{equation*}
\begin{split}
&\mathcal{Z}_0^{\pm}=\left\{z\in \Z~:~  p/N\ge\pm z>0\right\},\\
&\mathcal{Z}_i^{\pm}=\left\{z\in \Z~:~ \min\{p/2, e^{i}p/N\}\ge \pm z>e^{i-1}p/N\right\},\quad i=1, \ldots, I.
\end{split}
\end{equation*}
Then
\begin{equation}
\label{eq:T2}
%%   \sum_{m\in \Z_p}   \left| \sum_{z\in\Z_p^*-x}  \eta_z \ep\(m(x+z)^{-k})\)\right|
W(x)  \ll \sum_{i=0}^I |T_{i}^\pm(x)|,
\end{equation}
where
$$
T_{i}^\pm(x) =\sum_{m\in \Z_p}   \left| \sum_{z\in (\Z_p^*-x)\cap\mathcal{Z}_i^{\pm}}  \eta_z \ep\(m(x+z)^{-k})\)\right|,\qquad i=0, \ldots, I.
$$

For a fixed positive integer $\nu$, using again the H\"{o}lder inequality, we obtain
\begin{equation*}
\begin{split}
|T_{i}(x)^\pm|^{2\nu} & \leq p^{2\nu-1}\sum_{m\in\mathbb{Z}_{p}}\left|\sum_{\substack{z \in (\Z_p^*-x)\cap\mathcal{Z}_i^{\pm}}}
\eta_z  \ep(m (x+z)^{-k})\right|^{2\nu}\\
& = p^{2\nu-1}\mathop{\mathop{\sum\cdot\cdot\cdot\sum}_{z_{1},\ldots,z_{2\nu}\in (\Z_p^*-x)\cap\mathcal{Z}_i^{\pm}}}
%_{\flq{x_{i}^\ell} \le q/N, i=1,\ldots,2\nu}
\prod_{j=1}^{\nu}\eta_{z_{j}}\overline{\eta_{z_{\nu+j}}} \\
&\qquad\sum_{m\in\mathbb{Z}_{p}}\ep \(m \sum_{j=1}^{\nu}((x+z_{j})^{-k}-(x+z_{\nu+j})^{-k})\).
\end{split}
\end{equation*}
So, denoting by $\Omega_i^\pm(x)$ the following set
$$
\Omega_i^\pm(x)
=\left\{\(z_{1},\ldots,z_{2\nu}\) \in \(\cZ_i^{\pm}\)^{2\nu}
~:~\sum_{j=1}^{2\nu} \frac{(-1)^j }{(x+z_{j})^{k}} \equiv 0\pmod p
\right\}
$$
we can now see that
\begin{equation*}
\begin{split}
|T_{i}^\pm(x) |^{2\nu} & \leq  p^{2\nu}\sum_{(z_{1},\ldots,z_{2\nu})\in\Omega_i^\pm(x)}\prod_{j=1}^{\nu}\eta_{z_{j}}\overline{\eta_{z_{\nu+j}}}\\
& \ll (e^{-i}N)^{2\nu} p^{2\nu}  \# \Omega_i^\pm(x),
\end{split}
\end{equation*}
where we have used
$$
|\eta_z| \ll  e^{-i}N, \qquad i=0, \ldots, I.
$$

Applying  Lemma~\ref{lem:Ikl shift}, we have
\begin{equation*}
\begin{split}
|T_{i}^\pm(x) |^{2\nu} &\leq (e^{-i}N)^{2\nu}p^{2\nu+o(1)} \( \(e^i\frac{p}{N}\)^{2\nu}p^{-1}+ \(e^i\frac{p}{N}\)^{2\nu^{2}/(\nu+1)}\)\\
&= \(p^{4\nu-1}+p^{(4\nu^2 + 2\nu )/(\nu+1)}N^{2\nu/(\nu+1)}e^{-2\nu i/(\nu+1)}\)p^{o(1)}.
\end{split}
\end{equation*}
Then
\begin{equation}
\label{eq:T3}
|T_{i}^\pm(x) |\leq  \(p^{(4\nu-1)/(2\nu)}+p^{(2\nu +1)/(\nu+1)}N^{1/(\nu+1)}e^{-i/(\nu+1)}\)p^{o(1)}.
\end{equation}
The result now follows from~\eqref{eq:T1},  \eqref{eq:T2} and \eqref{eq:T3}.

\subsection{Proof of Theorem~\ref{thm:SAJq Aver}}

We consider only the integers  $q \in [Q, 2Q]$ for which the bound of  Corollary~\ref{cor:Ikl Aver} holds. Now,  proceeding as in the proof of the second bound in Theorem~\ref{thm:SAJq} but use Corollary~\ref{cor:Ikl Aver} instead of Lemma~\ref{lem:Ikl}.

\section{Comparison with previous results}
\label{sec:comp}

We note that for a prime $q=p$,
in our notation for  the functions $K_1(t)$ and $K_2(t)$ from~\cite{Nun1}
we have
$$
 K_1(mn) =  p^{-1/2} S_{-2,1,p}(ab^{2} m^2,n)=  p^{-1/2}  S_{2,-1,p}(ab^{2} m^2,n) $$
and
$$
 K_2(mn^2) =  p^{-1/2} S_{-2,1,p}(ab^{2} m,n) =
 p^{-1/2} S_{2,-1,p}(ab^{2} m,n).
$$
Recalling the definition~\eqref{eq:Special SAMN}, we now see that
the results of Nunes~\cite{Nun1}
can be written as
\begin{equation}
\label{eq:Nun T.1.2}
\cS_{2,p}^*(\cA;\cM_1,\cJ) \le
\sqrt{ \|\cA\|_{1}  \|\cA \|_{2}} p^{3/4+o(1)} M^{1/16} N^{5/8}
\end{equation}
provided that $1 \le M \le N^2$ and $MN^2 \le p^2$
and also
\begin{equation}
\label{eq:Nun T.1.3}
\cS_{2,p}^*(\cA; \cM_2,\cJ) \le
\sqrt{ \|\cA\|_{1}  \|\cA \|_{2}} p^{3/4+o(1)} M^{1/12} N^{7/12}
\end{equation}
provided that $1 \le M \le N^2$ and $MN\le p^{3/2}$,
where
$$\cM_1 =\{\alpha j~:~j =1, \ldots, M \} \mand
\cM_2 =  \{\alpha j^{2}~:~j =1, \ldots, M \}
$$
(with some $\alpha \in \F_p^*$)
and $\cJ$ is an interval of length $N<p$.
Using  Theorem~\ref{thm:SAJq} with $\nu = 2$ (and recalling that its bound does not depend
on the support $\cM$ of the weights $\cA$, see~\eqref{eq:Simple SAJ}), we obtain
$$
\cS_{2,p}^*(\cA;\cM_j,\cJ) \le
\sqrt{ \|\cA\|_{1}  \|\cA \|_{2}} \(p+ p^{11/12} N^{1/3}\) p^{o(1)}, \qquad j =1,2.
$$
This bound improves~\eqref{eq:Nun T.1.2} for
$$
MN^{10} \ge p^{4+\varepsilon} \mand M^3N^{14} \ge p^{8+\varepsilon}
$$
 and
 improves~\eqref{eq:Nun T.1.3} for
$$
MN^{7} \ge p^{3+\varepsilon}  \mand MN^{3} \ge p^{2+\varepsilon}
$$
with some fixed $\varepsilon > 0$.
In particular, if $M$ and $N$ are of similar sizes, that is, $N = M^{1+o(1)}$,
this happens for $M \ge p^{8/17 + \varepsilon}$ and $M \ge p^{1/2 + \varepsilon}$,
respectively.

We further note that for applications to smooth
numbers in arithmetic progressions only the bound~\eqref{eq:Nun T.1.2}
matters and only in the case of constant weights and thus it has to be
compared with that of  Corollary~\ref{cor:SIJq}
(it is easy to see that for
$\ell = 1$ it can be extended to the set $\cM_1 = \alpha \cI$).
 In particular, in this case the bound~\eqref{eq:Simple Bound}
is better when
$$
M^{13} N^{-2} \ge p^{4+\varepsilon}  \mand M^{23} N^{-6} \ge p^{8+\varepsilon},
$$
or similarly with $M$ and $N$ can be interchanged,
see also~\eqref{eq:Simple Bound Range} for the range when it is nontrivial.

In particular, in the critical for applications regime, when $N = M^{1+o(1)}$,
the bound~\eqref{eq:Simple Bound} is both better and nontrivial for
$ M\ge p^{8/17+\varepsilon}$.  However, the potential improvement of~\cite{Nun1},
which is implied by our bounds seems to be of the same strength as in the
follow up work of Nunes~\cite{Nun2}, where this is achieved via a different approach.

\section*{Acknowledgement}

The authors are grateful to Ramon Nunes for very useful discussions,
in particular for the information about his results in~\cite{Nun2}
and their comparison with potential improvements coming from our new bounds.

The first and the third authors gratefully acknowledge the support,   hospitality
and   excellent conditions of the School of Mathematics and Statistics of UNSW during their visit.

This work was supported by NSFC Grant 11401329 (for K.~Liu), by ARC Grant~DP140100118
(for I.~E.~Shparlinski) and by the Natural Science Foundation of Shaanxi Province of China Grant~2016JM1017
(for T.~P.~Zhang).


\begin{thebibliography}{9999}


\bibitem{BFKMM1} V. Blomer, {\'E}. Fouvry, E. Kowalski, P. Michel and D. Mili{\'c}evi{\'c},
 `On moments of twisted $L$-functions',
 {\it Amer. J. of Math.\/}, (to appear).

\bibitem{BFKMM2} V. Blomer, {\'E}. Fouvry, E. Kowalski, P. Michel and D. Mili{\'c}evi{\'c},
 `Some applications of smooth bilinear forms with Kloosterman sums',
  {\it Proc. Steklov Math. Inst.\/}, (to appear).

\bibitem{BouGar} J. Bourgain and M. Z. Garaev,
`Sumsets of reciprocals in prime fields and multilinear Kloosterman sums',
{\it Izv. Ross. Akad. Nauk Ser. Mat.\/}, {\bf 78} (2014), 9--72 (in Russian); translation in
{\it Izv. Math\/}, {\bf  78} (2014), 656--707.


\bibitem{FKM} {\'E}. Fouvry,  E. Kowalski and
  P.  Michel, `Algebraic trace functions over the primes',
{\it Duke Math. J.\/}, ,{\bf 163} (2014),  1683--1736.

%
%\bibitem{FKMRRS} {\'E}. Fouvry,  E. Kowalski,
%  P.  Michel,  C. S. Raju, J. Rivat and  K. Soundararajan,
% `On short sums of trace functions',
% {\it Annales de l'Institut Fourier\/} (to appear).


\bibitem{FouShp} \'E. Fouvry and I. E. Shparlinski,
`On a ternary quadratic form over
primes', {\it Acta Arith.\/}, {\bf 150} (2011), 285--314.


\bibitem{FrIw}
J.~B.~Friedlander and H.~Iwaniec,
`The divisor problem for arithmetic progressions',
{\it Acta Arith.\/}, \textbf{45} (1985), 273--277.

%\bibitem{Gar}
% M.Z.~Garaev,
%`An estimate of Kloosterman sums with prime numbers and application',
%{\it Matem. Zametki\/}, {\bf 88} (2010), 365--373 (in Russian); translation in
%{\it Math. Notes\/}, {\bf 88} (2010), 330?37.

%\bibitem{H-B1}
%D. R.~Heath-Brown,
%`Almost primes in arithmetic progressions and short intervals',
%{\it Math. Proc. Cambridge Philos. Soc.\/}, {\bf 83}  (1978),  357--375.

\bibitem{H-B2}
D. R.~Heath-Brown,
`The least square-free number in an arithmetic progression',
{\it J. Reine Angew. Math.\/}, {\bf  332} (1982), 204--220.

\bibitem{Hux} M. N. Huxley,
`A note on polynomial congruences',
{\it Recent Progress in Analytic Number Theory, Vol.1\/},  Academic
Press, 1981, 193--196.

 \bibitem{IwKow} H.  Iwaniec and E.  Kowalski,
{\it Analytic number theory\/}, American Mathematical Society,
Providence, RI, 2004.

\bibitem{Kar} A. A. Karatsuba, `Analogues of Kloosterman sums',
{\it  Izv. Ross. Akad. Nauk Ser. Mat. (Transl. as Russian Acad. Sci. Izv.
Math.)\/}, {\bf 55}(5) (1995),
93--102 (in Russian)


 \bibitem{Khan} R. Khan, `The divisor function in arithmetic progressions modulo prime powers',
{\it Mathematika\/},  {\bf 62} (2016), 898--908.

\bibitem{KMS} E. Kowalski,  P.  Michel and  W. Sawin,
 `Bilinear forms with Kloosterman sums and applications',
{\it Preprint\/}, 2015 (available from \url{http://arxiv.org/abs/1511.01636}).

\bibitem{LSZ} K. Liu, I. E. Shparlinski and T. P. Zhang,
 `Divisor problem in arithmetic progressions modulo a prime power',
{\it Preprint\/}, 2016  (available
from \url{http://arxiv.org/abs/1602.03583}).

\bibitem{Nun1}  R. M. Nunes,
 `Squarefree numbers in large arithmetic progressions',
{\it Preprint\/}, 2016 (available from \url{http://arxiv.org/abs/1602.00311}).

\bibitem{Nun2}  R. M. Nunes,
 `A note on the least squarefree number in an arithmetic progressions',
{\it Preprint\/}, 2016 (available from \url{http://arxiv.org/abs/1605.03347}).

\bibitem{Pierce}
L. B. Pierce, `The $3$-part of class numbers of quadratic fields,
{\it J. London Math. Soc.\/},  {\bf 71} (2005), 579--598.

\bibitem{RNRS}  O. Roche-Newton, M. Rudnev and I. D. Shkredov, `New sum-product
type estimates over finite fields', {\it Adv. Math.\/}, {\bf 293} (2016), 589--605.

\bibitem{Shp1} I. E. Shparlinski,
`On exponential sums with sparse
polynomials and rational functions',
{\it J. Number  Theory\/}, {\bf 60} (1996), 233--244.

\bibitem{Shp2} I. E. Shparlinski, `Bilinear forms with Kloosterman and Gauss sums',
{\it Preprint\/}, 2016  (available
from \url{http://arxiv.org/abs/1608.06160}).

\bibitem{ShpZha} I. E. Shparlinski and T. P. Zhang,
`Cancellations amongst Kloosterman sums',
{\it Acta Arith.\/}, {\bf  176} (2016), 201--210.

\bibitem{SS} D. Suryanarayana  and  R. Sitaramachandra Rao,
`The distribution of square-full integers', {\it Arkiv f{\"o}r
Matematik\/},  {\bf 11} (1973), 195--201.

\bibitem{WuXi} J. Wu and   P. Xi,
 `Arithmetic exponent pairs for algebraic trace functions and applications',
{\it Preprint\/}, 2016  (available
from \url{http://arxiv.org/abs/1603.07060}).

\bibitem{Xi-FKM}  P. Xi (with an appendix by  {\'E}. Fouvry,  E. Kowalski and
  P.  Michel),
 `Large sieve inequalities for algebraic trace functions',
{\it Intern. Math.  Res. Notices\/}, (to appear).

\bibitem{Vin}
I. M. Vinogradov, {\it Elements of number theory\/}, Dover Publ., NY, 1954.

\end{thebibliography}
\end{document}